\documentclass[12pt]{article}

\usepackage{amsmath}
\usepackage{amsfonts}
\usepackage{amssymb}
\usepackage{wasysym}
\usepackage{amsthm}
\usepackage[margin=1in]{geometry}

\title{Integral forms for tensor powers of the Virasoro vertex operator algebra
$L(\frac{1}{2},0)$ and their modules}
\author{Robert McRae}
\date{}

    \theoremstyle{definition}\newtheorem{rema}{Remark}[section]
    \theoremstyle{plain}\newtheorem{propo}[rema]{Proposition}
    \newtheorem{theo}[rema]{Theorem}
    \newtheorem{defi}[rema]{Definition}
    \newtheorem{lemma}[rema]{Lemma}
    \newtheorem{corol}[rema]{Corollary}
    \theoremstyle{definition}\newtheorem{exam}[rema]{Example}

\begin{document}
\bibliographystyle{alpha}
\maketitle

\newcommand{\Z}{\mathbb{Z}}
\numberwithin{equation}{section}

\begin{abstract}
\noindent We construct integral forms containing the conformal vector $\omega$ 
in certain tensor powers of the Virasoro vertex operator algebra 
$L(\frac{1}{2},0)$, and we construct integral forms in certain modules for these 
algebras. When a triple of modules for a tensor power of $L(\frac{1}{2},0)$ have 
integral forms, we classify which intertwining operators among these modules 
respect the integral forms. As an application, we explore how these results 
might be used to obtain integral forms in framed vertex operator algebras.
\end{abstract}

\section{Introduction}

The Virasoro vertex operator algebra $L(\frac{1}{2},0)$ is significant because 
many vertex operator algebras (called framed vertex operator algebras in 
\cite{DGH}) contain tensor powers of $L(\frac{1}{2},0)$ as vertex operator 
subalgebras. This class of framed vertex operator algebras most notably includes 
the lattice vertex operator algebra $V_{E_8}$ based on the $E_8$ root lattice and the moonshine module $V^\natural$ (\cite{B}, \cite{FLM2}), which contain $L(\frac{1}{2},0)^{\otimes 16}$ and $L(\frac{1}{2},0)^{\otimes 48}$, respectively (\cite{DMZ}). Thus, 
it is natural to look for integral forms in tensor powers of $L(\frac{1}{2},0)$ 
and its modules, and it is especially natural to look for integral forms which 
contain the conformal vector $\omega$. 

Such integral forms are worth studying in 
and of themselves as further examples of vertex algebras over $\Z$, beyond ones 
coming from lattice (\cite{B}, \cite{P}, \cite{DG}, \cite{M2}, \cite{M3}) and 
affine Lie algebra (\cite{GL}, \cite{M2}, \cite{M3}) vertex operator algebas. 
Integral forms in tensor powers of $L(\frac{1}{2},0)$ and their modules will also allow for the study of these vertex operator algebras over fields of prime characteristic. (The representation theory of $L(\frac{1}{2},0)$ itself over fields of odd prime characteristic has already been studied in \cite{DR1} and \cite{DR2}.) As another application, integral forms in tensor powers of $L(\frac{1}{2},0)$ and its modules 
may help construct interesting integral forms in framed vertex operator algebras, such as the moonshine module $V^\natural$. As indicated in the last section of this paper, if the irreducible modules appearing in the decomposition of a framed vertex operator algebra $V$ as an $L(\frac{1}{2},0)^{\otimes n}$-module have integral forms, and the vertex operator on $V$ induces intertwining operators among these modules which respect the integral forms, then $V$ has an integral form. The most substantial difficulty obstructing the implementation of this idea seems to be that the intertwining operators among the $L(\frac{1}{2},0)^{\otimes n}$-modules appearing in the decomposition of $V$ may in fact map the proposed integral form of $V$ into a different, larger, integral form.

At first glance, it might not seem natural to look for integral forms in 
$L(\frac{1}{2},0)$, since the Virasoro algebra commutation relations which 
underlie its vertex operator algebra structure,
\begin{equation*}
 [L(m), L(n)]=(m-n)L(m+n)+\dfrac{1}{4}\binom{m+1}{3}\delta_{m+n,0}
\end{equation*}
for $m,n\in\Z$, are not integral. In fact, we show that $L(\frac{1}{2},0)$ does 
have an integral form generated by $2\omega$, though it is impossible for an 
integral form of $L(\frac{1}{2},0)$ to contain $\omega$ itself. To find integral 
forms which do contain $\omega$, we must consider tensor powers 
$L(\frac{1}{2},0)^{\otimes n}$ where $n\in 4\Z$, because the Virasoro algebra 
commutation relations become integral when the central charge is an even 
integer.

The first main result in this paper is that when $n\in 4\Z$, 
$L(\frac{1}{2},0)^{\otimes n}$ indeed has integral forms whose generators (which 
include $\omega$) are indexed by the elements of binary linear codes which 
satisfy certain easily checked conditions. Moreover, we show that in some 
circumstances, $L(\frac{1}{2},0)$-modules whose conformal weights are integral 
have integral forms. In particular, we show that $L(\frac{1}{2},0)^{\otimes 16}$ 
has an integral form such that all irreducible submodules which appear in the 
decomposition of $V_{E_8}$ as an $L(\frac{1}{2},0)^{\otimes 16}$-module have 
integral forms. To prove these results, we heavily use the philosophy of \cite{M2} that integral forms are often best constructed and studied using generating sets; thus we do not need to find bases of the integral forms we construct in order to prove that they exist and contain desired elements such as $\omega$.

Recall that constructing modules is not enough in order to understand the 
representation theory of a vertex operator algebra $V$. Two important elements 
of the representation theory of $V$ are contragredient modules and intertwining 
operators; for example, these are needed for the construction of braided tensor 
categories of $V$-modules (see for example the review article \cite{HL}). Thus in 
this paper, we consider graded $\Z$-duals of and intertwining operators among 
the integral forms of $L(\frac{1}{2},0)^{\otimes n}$-modules that we construct. 
It is easy to use results in \cite{M2} to show that graded 
$\Z$-duals of the integral forms in this paper are integral forms in contragredient modules. The result for intertwining operators is more subtle, and we discuss it in more detail.

The second main result of this paper is a classification of intertwining operators among a triple $W^{(1)}$, $W^{(2)}$, $W^{(3)}$ of $L(\frac{1}{2},0)^{\otimes n}$-modules which respect certain integral forms in the modules. More particularly, suppose $\mathcal{C}$ is the binary linear code used to construct the integral form in $L(\frac{1}{2},0)^{\otimes n}$; we use $W^{(i)}_\mathcal{C}$ for $i=1,2,3$ to denote the corresponding integral forms in the modules $W^{(i)}$. Since $L(\frac{1}{2},0)^{\otimes n}$-modules are self-contragredient, the graded $\Z$-duals $(W^{(i)}_\mathcal{C})'$ are different integral forms of the modules $W^{(i)}$. Then we determine which intertwining operators $\mathcal{Y}$ satisfy
\begin{equation*}
 \mathcal{Y}: W^{(1)}_\mathcal{C}\otimes W^{(2)}_{\mathcal{C}}\rightarrow (W^{(3)}_\mathcal{C})'\lbrace x\rbrace,
\end{equation*}
or equivalently, which satisfy
\begin{equation*}
 \langle w_{(3)}, \mathcal{Y}(w_{(1)},x)w_{(2)}\rangle\in\Z\lbrace x\rbrace
\end{equation*}
for $w_{(i)}\in W^{(i)}_\mathcal{C}$. This last integrality conditions also arose in \cite{M3} in the study of intertwining operators among lattice and affine Lie algebra vertex operator algebras, and it is natural because it amounts to an integrality condition on physically relevant correlation functions in conformal field theory. It is not clear that there are generally any non-zero intertwining operators which satisfy
\begin{equation*}
 \mathcal{Y}: W^{(1)}_\mathcal{C}\otimes W^{(2)}_{\mathcal{C}}\rightarrow W^{(3)}_\mathcal{C}\lbrace x\rbrace,
\end{equation*}
and it is this subtlety which obstructs the construction of integral forms in framed vertex operator algebras using the results in this paper.

We also note that to classify integral intertwining operators among $L(\frac{1}{2},0)^{\otimes n}$-modules, we need to use ideas motivated by the cross-brackets of vertex operators introduced in \cite{FLM2} Section 8.9. In \cite{FLM2} it was shown that cross-brackets of vectors of conformal weight $2$ behave more naturally than for instance commutators do, which explains why related ideas are useful here in our study of integral forms generated by vectors of conformal weight $2$.

We now outline the contents of this paper. Sections 2, 3, and 4 are preliminary: 
Section 2 recalls definitions and results from \cite{M2} and \cite{M3} on 
integral forms of vertex operator algebras that we will need; Section 3 recalls 
the construction of tensor products of vertex operator algebras and their 
modules; and Section 4 recalls the construction of vertex operator algebras 
based on the Virasoro algebra, and also proves that the irreducible Virasoro 
algebra module $L(\ell, h)$ has a $\mathbb{Q}$-form when the central charge 
$\ell$ and lowest conformal weight $h$ are rational. Section 5 obtains integral 
forms in tensor powers $L(\frac{1}{2},0)^{\otimes n}$, which are based on binary 
linear codes and contain $\omega$ when $n\in 4\Z$; some examples of codes that 
give rise to these integral forms are also given. Section 6 obtains integral 
forms in certain $L(\frac{1}{2},0)^{\otimes n}$-modules which have integral 
conformal weights, and Section 7 treats graded $\Z$-duals of and integral 
intertwining operators among integral forms of $L(\frac{1}{2},0)^{\otimes 
n}$-modules. In Section 8, we suggest an approach to obtaining integral forms in 
framed vertex operator algebras using results from the previous sections.

In this paper, we assume some level of familiarity with the theory of vertex 
operator algebras, and so we will freely use standard vertex algebra notation 
and terminology throughout.

\paragraph{Acknowledgments}
This paper is part of my thesis \cite{M1}, completed at Rutgers University. I am 
very grateful to my advisor James Lepowsky for all of his advice and 
encouragement.

\section{Integral forms in vertex operator algebras and modules}

 We use the notions of vertex operator algebra and module for a vertex operator 
algebra as defined in \cite{FLM2}; see also \cite{LL}. We use the definition of an integral form in 
a vertex operator algebra and its modules from \cite{M2}:
\begin{defi}
 An \textit{integral form} of a vertex operator algebra $V$ is a vertex subring 
$V_\Z\subseteq V$ that is an integral form of $V$ as a vector space and that is 
compatible with the conformal weight grading of $V$:
 \begin{equation}\label{compatibility}
  V_\Z=\coprod_{n\in\mathbb{Z}} V_{(n)}\cap V_\Z,
 \end{equation}
where $V_{(n)}$ is the conformal weight space with $L(0)$-eigenvalue $n$. An 
\textit{integral form} in a $V$-module $W$ is a $V_\Z$-submodule $W_\Z\subseteq 
W$ that is an integral form of $W$ as a vector space and that is compatible with 
the conformal weight grading of $W$:
 \begin{equation}\label{modcompatibility}
  W_\mathbb{Z}=\coprod_{h\in\mathbb{C}} W_{(h)}\cap W_\Z,
 \end{equation}
where $W_{(h)}$ is the weight space with $L(0)$-eigenvalue $h$.
\end{defi}
\begin{rema}
 We could equivalently define an integral form $V_\Z$ of $V$ as the 
$\mathbb{Z}$-span of a basis for $V$ which contains the vacuum $\mathbf{1}$, is 
closed under vertex algebra products, and is compatible with the conformal 
weight gradation. Similarly, we could define an integral form $W_\Z$ of a 
$V$-module $W$ as the $\mathbb{Z}$--span of a basis for $W$ which is preserved 
by vertex operators from $V_\Z$ and is compatible with the conformal weight 
gradation. Note that $V$ may have more than one integral form, and the notion of 
integral form in a $V$-module $W$ depends on which integral form of $V$ we use.
\end{rema}
\begin{rema}
 The definition of an integral form $V_\Z$ of a vertex operator algebra $V$ 
given in \cite{DG} is somewhat different, since it replaces compatibility with 
the weight gradation with the requirement that $V_\Z$ contain an integral 
multiple of the conformal vector $\omega$. The examples studied in this paper will satisfy both conditions.
\end{rema}

As in \cite{M2} and \cite{M3}, we use generating sets to study integral forms in 
vertex operator algebras and modules. We recall that the vertex subalgebra of a 
vertex operator algebra $V$ generated by a subset $S$ is the smallest vertex 
subalgebra of $V$ containing $S$, and similarly the submodule of a $V$-module 
$W$ generated by a subset $T$ is the smallest $V$-submodule of $W$ containing 
$T$. These notions immediately carry over to vertex algebras over $\Z$ and their 
modules. We will need the following proposition from \cite{M2}:
\begin{propo}\label{zgen}
 Suppose $V$ is a vertex algebra; for a subset $S$ of $V$, denote by 
$\left\langle S\right\rangle_\mathbb{Z}$ the vertex subring
generated by $S$. Then  $\left\langle S\right\rangle_\mathbb{Z}$ is the 
$\mathbb{Z}$-span of coefficients of products of the form
\begin{equation}\label{zspan}
 Y(u_1,x_1)\ldots Y(u_k,x_k)\mathbf{1}
\end{equation}
where $u_1,\ldots u_k\in S$. Moreover, if $W$ is a $V$-module, the 
$\left\langle S\right\rangle_\mathbb{Z}$-submodule generated by a subset 
$T$ of $W$ is the $\mathbb{Z}$-span of coefficients of products of the 
form
\begin{equation}
 Y(u_1,x_1)\ldots Y(u_k,x_k)w
\end{equation}
where $u_1,\ldots u_k\in S$ and $w\in T$.
\end{propo}

We will use the notion of the contragredient $W'$ of a $V$-module $W$ from 
\cite{FHL}. If $V$ has an integral form $V_\Z$ and $W$ has an integral form 
$W_\Z$, then $W'$ has an integral form as a vector space, the \textit{graded 
$\Z$-dual}
\begin{equation*}
 W'_\Z=\lbrace w'\in W'\,\vert\,\langle w',w\rangle\in\Z\,\,\mathrm{for}\,\,w\in 
W_\Z\rbrace.
\end{equation*}
The following propositions from \cite{M2} (see also Lemma 6.1, Lemma 6.2, and 
Remark 6.3 in \cite{DG}) give sufficient conditions  for $W_\Z'$ to be a 
$V_\Z$-module:
\begin{propo}\label{contmod}
 Suppose $V_\mathbb{Z}$ is preserved by $\frac{L(1)^n}{n!}$ for $n\geq 0$.
Then $W'_\mathbb{Z}$ is preserved by the action of $V_\mathbb{Z}$.
\end{propo}
\begin{propo}\label{invarunderl1}
 If $V_\mathbb{Z}$ is generated by vectors $v$ such that $L(1)v=0$, then 
$V_\mathbb{Z}$ is preserved by $\frac{L(1)^n}{n!}$ for $n\geq 0$.
\end{propo}

Sometimes a $V$-module $W$ is equivalent as a $V$-module to its contragredient 
$W'$; by Remark 5.3.3 in \cite{FHL} this happens exactly when $W$ has a 
nondegenerate bilinear form $(\cdot,\cdot)$ that is \textit{invariant}:
\begin{equation*}
 (Y(v,x)w',w)=(w',Y(e^{x L(1)} (-x^{-2})^{L(0)} v,x^{-1})w)
\end{equation*}
for all $v\in V$, $w,w'\in W$. In the special case $W=V$, we know  from 
\cite{Li} that the space of (not necessarily nondegenerate) invariant bilinear 
forms is linearly isomorphic to $V_{(0)}/L(1)V_{(1)}$. If $V$ is a simple vertex 
operator algebra, then any non-zero invariant bilinear form is necessarily 
nondegenerate.

 We will use the notion of intertwining operator among modules for a vertex 
operator algebra from \cite{FHL}. Suppose that $V$ is a vertex operator 
algebra with integral form $V_\Z$ and $W^{(i)}$ for $i=1,2,3$ are $V$-modules 
with integral forms $W^{(i)}_\Z$. If $\mathcal{Y}$ is an intertwining operator 
of type $\binom{W^{(3)}}{W^{(1)}\,W^{(2)}}$, recall from \cite{M3} that we 
say $\mathcal{Y}$ is \textit{integral} with respect to the integral forms 
$W^{(i)}_\Z$ if
 \begin{equation*}
  \mathcal{Y}(w_{(1)},x)w_{(2)}\in W^{(3)}_\Z\lbrace x\rbrace
 \end{equation*}
for all $w_{(1)}\in W^{(1)}_\Z$ and $w_{(2)}\in W^{(2)}_\Z$. We will need the 
following theorem from \cite{M3}, which states that to prove that an 
intertwining operator is integral, it is enough to check that it is integral on 
generators:
\begin{theo}\label{intwopgen}
 Suppose $V$ is a vertex operator algebra with integral form $V_\mathbb{Z}$ and 
$W^{(1)}$, $W^{(2)}$, and $W^{(3)}$ are $V$-modules with integral forms 
$W^{(1)}_\mathbb{Z}$, $W^{(2)}_\mathbb{Z}$, and $W^{(3)}_\mathbb{Z}$, 
respectively. Moreover, suppose $T^{(1)}$ and $T^{(2)}$ are generating sets for 
$W^{(1)}_\mathbb{Z}$ and $W^{(2)}_\mathbb{Z}$, respectively. If an intertwining 
operator $\mathcal{Y}$ of type 
$\binom{W^{(3)}}{W^{(1)}\, W^{(2)}}$ satisfies
 \begin{equation*}
  \mathcal{Y}(t_{(1)},x)t_{(2)}\in W^{(3)}_\mathbb{Z}\lbrace x\rbrace
 \end{equation*}
for all $t_{(1)}\in T^{(1)}$, $t_{(2)}\in T^{(2)}$, then $\mathcal{Y}$ is 
integral with respect to $W^{(1)}_\mathbb{Z}$, $W^{(2)}_\mathbb{Z}$, and 
$W^{(3)}_\mathbb{Z}$.
\end{theo}

We conclude this section by recalling some results from \cite{M2} on the 
conformal vector $\omega$ in an integral form. If $c\in\mathbb{C}$ is the 
central charge of a vertex operator algebra $V$ with integral form $V_\Z$, we 
have:
\begin{propo}\label{omegainintform1}
 If $V_\mathbb{Z}$ contains $k\omega$ where $k\in\mathbb{C}$, then $k^2 c\in
2\mathbb{Z}$.
\end{propo}
\noindent As a partial converse of Proposition \ref{omegainintform1} for vertex 
operator algebras generated by $\omega$, we have
\begin{propo}\label{viralgprop}
 If $V$ is a vertex operator algebra generated by the conformal vector 
$\omega$, 
and $\omega$ is contained in a rational form of $V$, 
then $V$ has an integral form generated by $k\omega$ if $k\in\mathbb{Z}$ and 
$k^2 c\in 2\mathbb{Z}$. In particular, $\omega$ generates an integral form of 
$V$ 
if and only if $c\in 2\mathbb{Z}$.
\end{propo}

\section{Tensor products of vertex operator algebras}
We now recall from \cite{FHL} the definition of the tensor product of vertex 
operator algebras $(U, 
Y_U,\mathbf{1}_U,\omega_U)$ and $(V,Y_V,\mathbf{1}_V,\omega_V)$. 
The tensor product vertex operator algebra is the vector space $U\otimes V$ 
with 
vertex operator given by
\begin{equation}\label{tensorprodop}
 Y_{U\otimes V} (u_{(1)}\otimes v_{(1)},x)(u_{(2)}\otimes 
v_{(2)})=Y_U(u_{(1)},x)u_{(2)}\otimes Y_V(v_{(1)},x)v_{(2)},
\end{equation}
vacuum given by
\begin{equation*}
 \mathbf{1}_{U\otimes V}=\mathbf{1}_U\otimes\mathbf{1}_V
\end{equation*}
and conformal vector given by
\begin{equation*}
 \omega_{U\otimes V}=\omega_U\otimes\mathbf{1}_V+\mathbf{1}_U\otimes\omega_V.
\end{equation*}
The central charge of $U\otimes V$ is the sum of the central charges of $U$ and 
$V$. If $W_U$ is a $U$-module and $W_V$ is a $V$-module, then $W_U\otimes W_V$ 
is a $U\otimes V$-module with vertex operator analogous to (\ref{tensorprodop}):
\begin{equation*}
 Y_{W_U\otimes W_V} (u\otimes v,x)(w_U\otimes 
w_V)=Y_{W_U}(u,x)w_U\otimes Y_{W_V}(v,x)w_V,
\end{equation*} 
We have the following result on generating sets for tensor product vertex 
operator algebras and modules:
\begin{propo}\label{tensorgens}
 If $U$ is generated by $S$ and $V$ is generated by $T$, then $U\otimes V$ 
is 
generated by
 \begin{equation*}
  \lbrace s\otimes\mathbf{1},\,\mathbf{1}\otimes t\,\vert\, s\in S,\,t\in 
T\rbrace.
 \end{equation*}
Moreover, if $W_U$ is a $U$-module generated by $Q$ and $W_V$ is a $V$-module 
generated by $R$, then $W_U\otimes W_V$ is generated as a $U\otimes V$-module by
\begin{equation*}
 \lbrace q\otimes r\,\vert\, q\in Q,\,r\in R\rbrace.
\end{equation*}
\end{propo}
\begin{proof}
 Since $U$ is generated by $S$ and $V$ is generated by $T$, the subalgebra of 
$U\otimes V$ generated by the elements $s\otimes\mathbf{1}$ and 
$\mathbf{1}\otimes t$ contains $U\otimes\mathbf{1}$ and $\mathbf{1}\otimes V$. 
Then this subalgebra must contain all of $U\otimes V$ because for any $u\in U$, 
$v\in V$, it contains
 \begin{equation*}
  \mathrm{Res}_x\, x^{-1} Y(u\otimes\mathbf{1},x)(\mathbf{1}\otimes 
v)=\mathrm{Res}_x\, x^{-1} Y(u,x)\mathbf{1}\otimes v=u\otimes v.
 \end{equation*}
Moreover, since $Q$ generates $W_U$ as a $U$-module, by applying vertex 
operators of the form $Y(u\otimes\mathbf{1},x)$ for $u\in U$ to vectors of the 
form $q\otimes r$, we see that the $U\otimes V$-submodule of $W_U\otimes W_V$ 
generated by the vectors $q\otimes r$ contains $W_U\otimes r$ for any $r\in R$. 
Then by applying vertex operators of the form $Y(\mathbf{1}\otimes v,x)$, we 
see 
that the submodule generated by the vectors $q\otimes r$ equals $W_U\otimes 
W_V$.
\end{proof}

 It is easy to see that if $U$ and $V$ are two vertex operator algebras with 
integral forms $U_\mathbb{Z}$ and $V_\mathbb{Z}$, respectively, then 
$U_\mathbb{Z}\otimes_\mathbb{Z} V_\mathbb{Z}$ is an integral form of $U\otimes 
V$. Moreover, if $W_U$ is a $U$-module and $W_V$ is a $V$-module with integral 
forms $(W_U)_\mathbb{Z}$ and $(W_V)_\mathbb{Z}$, respectively, then 
$(W_U)_\mathbb{Z}\otimes_\mathbb{Z} (W_V)_\mathbb{Z}$ is a 
$U_\mathbb{Z}\otimes_\mathbb{Z} V_\mathbb{Z}$-module. It is also clear that 
Proposition \ref{tensorgens} applies to vertex algebras over $\Z$ and their modules.

\begin{rema}
 All the definitions and results in this section have obvious generalizations 
to 
tensor products of more than two algebras or modules.
\end{rema}

\section{Vertex operator algebras based on the Virasoro algebra}
Now we recall the construction of vertex operator algebras based on the 
Virasoro 
algebra (see for instance \cite{FZ} or \cite{LL} Section 6.1 for more details). 
Recall the Virasoro Lie algebra
\begin{equation*}
 \mathcal{L}=\coprod_{n\in\mathbb{Z}}\mathbb{C} L_n\oplus\mathbb{C}\mathbf{c}
\end{equation*}
with $\mathbf{c}$ central and all other commutation relations given by
\begin{equation}\label{viralgcomm}
 [L_m, L_n]=(m-n)L_{m+n}+\frac{m^3-m}{12}\delta_{m+n,0}\mathbf{c}
\end{equation}
for any $m,n\in\mathbb{Z}$. The Virasoro algebra has the decomposition into 
subalgebras
\begin{equation*}
 \mathcal{L}=\mathcal{L}_+\oplus\mathcal{L}_0\oplus\mathcal{L}_-,
\end{equation*}
where
\begin{equation*}
 \mathcal{L}_\pm=\coprod_{n\in\mp\mathbb{Z}_+}\mathbb{C} L_n
\end{equation*}
and
\begin{equation*}
 \mathcal{L}_0=\mathbb{C} L_0\oplus\mathbb{C}\mathbf{c}.
\end{equation*}
We also define the subalgebra
\begin{equation*}
 \mathcal{L}_{\leq 1} =\mathcal{L}_-\oplus\mathcal{L}_0\oplus\mathbb{C} L_{-1}.
\end{equation*}
For any $\mathcal{L}$-module $V$, we use $L(n)$ to denote the action of $L_n$ 
on 
$V$.

Now for any complex number $\ell$, we have the one-dimensional 
$\mathcal{L}_{\leq 1}$-module $\mathbb{C}_\ell$ on which $\mathcal{L}_-$, 
$L_0$, 
and $L_{-1}$ act trivially and on which $\mathbf{c}$ acts as the scalar $\ell$. 
Then we form the induced module
\begin{equation*}
 V(\ell,0)=U(\mathcal{L})\otimes_{U(\mathcal{L}_{\leq 1})}\mathbb{C}_\ell,
\end{equation*}
which is a vertex operator algebra with vacuum $\mathbf{1}=1\otimes 1$ and 
generated by its conformal vector $\omega =L(-2)\mathbf{1}$ with vertex operator
\begin{equation*}
 Y(L(-2)\mathbf{1},x)=\sum_{n\in\mathbb{Z}} L(n) x^{-n-2}.
\end{equation*}

To construct irreducible $V(\ell,0)$-modules, we take a complex number $h$ 
and consider the one-dimensional $\mathcal{L}_-\oplus\mathcal{L}_0$-module 
$\mathbb{C}_{\ell, h}$ on which $\mathcal{L}_-$ acts trivially, $\mathbf{c}$ 
acts as the scalar $\ell$ and $L_0$ acts as the scalar $h$. Then we form the 
Verma module
\begin{equation*} 
M(\ell,h)=U(\mathcal{L})\otimes_{U(\mathcal{L}_-\oplus\mathcal{L}_0)}\mathbb{C}_
{\ell,h},
\end{equation*}
which is a $V(\ell,0)$-module. For any $h\in\mathbb{C}$, $M(\ell,h)$ has a 
unique irreducible quotient $L(\ell,h)$, and these modules $L(\ell,h)$ exhaust 
the irreducible $V(\ell,0)$-modules up to equivalence. 

Note that $V(\ell,0)$ itself is a quotient of $M(\ell,0)$ by the submodule 
generated by $L(-1)\mathbf{1}$. It is often the case that $V(\ell,0)$ is 
irreducible as a module for itself and is thus equal to $L(\ell,0)$. In this 
case, the irreducible $L(\ell,0)$-modules consist of all $L(\ell,h)$. From 
\cite{W}, $V(\ell,0)$ is reducible if and only if
\begin{equation*}
 \ell=c_{p,q}=1-\dfrac{6(p-q)^2}{pq}
\end{equation*}
where $p$ and $q$ are relatively prime integers greater than $1$. In this case, 
the irreducible $L(c_{p,q},0)$-modules are the modules $L(c_{p,q},h_{m,n})$ 
where
\begin{equation*}
 h_{m,n}=\frac{(np-mq)^2-(p-q)^2}{4pq}
\end{equation*}
for $0<m<p$ and $0<n<q$.

Taking $p=3$ and $q=4$, we obtain $c_{3,4}=\frac{1}{2}$, and $L(\frac{1}{2},0)$ 
has the three irreducible modules $L(\frac{1}{2},0)$, 
$L(\frac{1}{2},\frac{1}{2})$, and $L(\frac{1}{2},\frac{1}{16})$. Moreover, for 
any positive integer $n$, the vertex operator algebra 
$L(\frac{1}{2},0)^{\otimes 
n}$ is simple, and all its irreducible modules are obtained as tensor products 
of irreducible modules for $L(\frac{1}{2},0)$ (\cite{FHL}).

We conclude this section by showing the existence of $\mathbb{Q}$-forms inside 
the irreducible $\mathcal{L}$-modules $L(\ell,h)$ for $\ell,h\in\mathbb{Q}$, 
using an argument analogous to the one used to obtain $\Z$-forms in irreducible 
modules for finite-dimensional complex simple Lie algebras (Lemma 12 in 
\cite{S}; see also Theorem 27.1 in \cite{Hu}):
\begin{propo}\label{virmodqform}
 Suppose $\ell,h\in\mathbb{Q}$ and $v_h$ spans the lowest conformal weight 
space 
of $L(\ell,h)$. Then $L(\ell,h)$ has a $\mathbb{Q}$-form $L(\ell,h)_\mathbb{Q}$ 
which is the $\mathbb{Q}$-span of vectors of the form
 \begin{equation}\label{virmodspan}
  L(-n_1)\cdots L(-n_k) v_h
 \end{equation}
where $n_i>0$.
\end{propo}
\begin{proof}
 Let $U_\mathbb{Q}(\mathcal{L})$ denote the $\mathbb{Q}$-subalgebra of 
$U(\mathcal{L})$ spanned by monomials in the basis elements $L_n$ for 
$n\in\mathbb{Z}$ and $\mathbf{c}$. Similarly, let 
$U_\mathbb{Q}(\mathcal{L}_\pm)$ denote the $\mathbb{Q}$-subalgebras spanned by 
monomials in the basis elements $L_{\mp n}$ for $n>0$, and let 
$U_\mathbb{Q}(\mathcal{L}_0)$ denote the $\mathbb{Q}$-subalgebra spanned by 
monomials in $L_0$ and $\mathbf{c}$. Since the Lie algebra structure constants for the 
basis 
$\lbrace L_n\rbrace_{n\in\mathbb{Z}}\cup\lbrace\mathbf{c}\rbrace$ given by 
(\ref{viralgcomm}) are rational, choosing an appropriate order on the basis and 
the Poincar\'{e}-Birkhoff-Witt Theorem show that $U_\mathbb{Q}(\mathcal{L})$ is 
a $\mathbb{Q}$-form of the associative algebra $U(\mathcal{L})$ and
 \begin{equation*}  
U_{\mathbb{Q}}(\mathcal{L})=U_\mathbb{Q}(\mathcal{L}_+)U_\mathbb{Q}(\mathcal{L}
_0)U_\mathbb{Q}(\mathcal{L}_-).
 \end{equation*}

 Now we define the $U_\mathbb{Q}(\mathcal{L})$-module
\begin{equation*}
 L(\ell,h)_\mathbb{Q}=U_\mathbb{Q}(\mathcal{L})\cdot 
v_h=U_\mathbb{Q}(\mathcal{L}_+)\cdot v_h
\end{equation*}
where the second equality follows because $L(n)v_h=0$ for $n>0$, 
$L(0)v_h=hv_h\in\mathbb{Q} v_h$, and $\mathbf{c}\cdot v_h=\ell v_h\in\mathbb{Q} 
v_h$. Note that $L(\ell,h)_\mathbb{Q}$ is thus the $\mathbb{Q}$-span of 
vectors of the form (\ref{virmodspan}), which means that the $\mathbb{C}$-span 
of $L(\ell,h)_\mathbb{Q}$ is all of $L(\ell,h)$. Note also from \eqref{virmodspan} that the intersection 
of $L(\ell, h)_\mathbb{Q}$ with the lowest conformal weight space 
$\mathbb{C}v_h$ of $L(\ell,h)$ is $\mathbb{Q}v_h$ and that $L(\ell, h)_\mathbb{Q}$ is graded by conformal weight. For any $w\in L(\ell,h)$, we 
use $w_h$ to denote the component of $w$ in the lowest conformal weight space 
of 
$L(\ell,h)$.

Since $L(\ell,h)_\mathbb{Q}$ is the $\mathbb{Q}$-span of a spanning set for 
$L(\ell,h)$ over $\mathbb{C}$, to show that $L(\ell,h)_\mathbb{Q}$ is a 
$\mathbb{Q}$-form of $L(\ell,h)$, we just need to show that any set of vectors 
in $L(\ell,h)_\mathbb{Q}$ that is linearly independent over $\mathbb{Q}$ is 
also 
linearly independent over $\mathbb{C}$. To prove this, suppose to the contrary 
that $\lbrace w_i\rbrace_{i=1}^k\subseteq L(\ell,h)_\mathbb{Q}$ is a set of 
(non-zero) vectors of minimal cardinality which are linearly independent over 
$\mathbb{Q}$ but linearly dependent over $\mathbb{C}$ (note that $k\geq 2$). 
This means there is some dependence relation
\begin{equation*}
 \sum_{i=1}^k c_i w_i=0
\end{equation*}
where $c_i\in\mathbb{C}^\times$.

Now, there is some $y\in U_\mathbb{Q}(\mathcal{L}_-)$ such that $(y\cdot 
w_1)_h\neq 0$, because otherwise $w_1$ generates a proper 
$\mathcal{L}$-submodule of $L(\ell,h)$, contradicting the irreducibility of 
$L(\ell,h)$. Note that for $1\leq i\leq k$, $(y\cdot w_i)_h=q_i v_h$ where 
$q_i\in\mathbb{Q}$ and $q_1\neq 0$, since $y$ preserves $L(\ell,h)_\mathbb{Q}$, $L(\ell,h)_\mathbb{Q}$ is graded by conformal weight,
and $L(\ell,h)_\mathbb{Q}\cap\mathbb{C} v_h=\mathbb{Q} v_h$. Thus
\begin{equation*}
 0=\left(y\cdot\sum_{i=1}^k c_i w_i\right)_h=\sum_{i=1}^k c_i (y\cdot 
w_i)_h=\left(\sum_{i=1}^k c_i q_i\right) v_h,
\end{equation*}
or $\sum_{i=1}^k c_i q_i=0$. Then
\begin{equation*}
 0=q_1\sum_{i=1}^k c_i w_i-\left(\sum_{i=1}^k c_i q_i\right) w_1=\sum_{i=1}^k 
(c_i q_1 w_i-c_i q_i w_1)=\sum_{i=2}^k c_i(q_1 w_i-q_i w_1).
\end{equation*}
Since $q_1\neq 0$ and the vectors $\lbrace w_i\rbrace_{i=1}^k$ are linearly 
independent over $\mathbb{Q}$, the vectors $\lbrace q_1 w_i-q_i 
w_1\rbrace_{i=2}^k$ are also linearly indepedent over $\mathbb{Q}$. But since 
they are also dependent over $\mathbb{C}$, this contradicts the minimality of 
$\lbrace w_i\rbrace_{i=1}^k$.
\end{proof}

\section{Integral forms in $L(\frac{1}{2},0)^{\otimes n}$}
In this section we construct integral forms in tensor powers of the Virasoro 
vertex operator algebra $L(\frac{1}{2},0)$ using generators. We start with the 
following basic result:
\begin{propo}\label{badform}
 The vertex operator algebra $L(\frac{1}{2},0)$ has an integral form generated 
by $2\omega$.
\end{propo}
\begin{proof}
 Since $\omega=L(-2)\mathbf{1}$ is contained in the $\mathbb{Q}$-form 
$L(\frac{1}{2},0)_\mathbb{Q}$ given by Proposition \ref{virmodqform} and 
$\omega$ generates $L(\frac{1}{2},0)$ as a vertex operator algebra, Proposition 
\ref{viralgprop} immediately implies that $L(\frac{1}{2},0)$ has an integral 
form generated by $2\omega$. 
\end{proof}

Now for any positive integer $n$, we consider $L(\frac{1}{2},0)^{\otimes n}$; 
for any $i$ such that $1\leq i\leq n$, we set
\begin{equation*}
 \omega^{(i)}=\mathbf{1}^{\otimes (i-1)}\otimes\omega\otimes\mathbf{1}^{\otimes 
(n-i)}.
\end{equation*}
Then as an immediate consequence of Propositions \ref{tensorgens} and 
\ref{badform}, we have
\begin{corol}
 For any integer $n\geq 1$, $L(\frac{1}{2},0)^{\otimes n}$ has an integral form 
generated by the vectors $2\omega^{(i)}$ for $1\leq i\leq n$.
\end{corol}
Although the preceding corollary gives integral forms in 
$L(\frac{1}{2},0)^{\otimes n}$ for any $n$, these forms do not contain 
$\omega$. 
We would like to construct integral forms in $L(\frac{1}{2},0)^{\otimes n}$ 
which do contain $\omega$; from Proposition \ref{omegainintform1}, this will 
not 
be possible unless $n\in 4\mathbb{Z}$, since the central charge of 
$L(\frac{1}{2},0)^{\otimes n}$ is $\frac{n}{2}$. To facilitate the construction 
of integral forms containing $\omega$ in $L(\frac{1}{2},0)^{\otimes n}$, we 
recall the definition of a binary linear code (see for example \cite{MS}, but 
we 
use the notation of \cite{FLM2} Chapter 10).

Let $\Omega$ be an $n$-element set. Then the power set $P(\Omega)$ is a vector 
space over $\mathbb{F}_2$, the field of two elements, where addition is given 
by 
symmetric difference: for $S,T\subseteq\Omega$,
\begin{equation*}
 S+T=(S\setminus(S\cap T))\cup (T\setminus (S\cap T)).
\end{equation*}
A \textit{binary linear code} $\mathcal{C}$ is an $\mathbb{F}_2$-subspace of 
$P(\Omega)$. One example of a binary linear code is the subspace
\begin{equation*}
 \mathcal{E}(\Omega)=\lbrace S\subseteq\Omega\,\vert\, \vert S\vert\in 
2\mathbb{Z}\rbrace
\end{equation*}
of subsets of $\Omega$ of even cardinality. 

The $\mathbb{F}_2$-vector space $P(\Omega)$ has a nondegenerate 
$\mathbb{F}_2$-valued bilinear form given by
\begin{equation*}
 (S,T)\mapsto \vert S\cap T\vert +2\mathbb{Z}
\end{equation*}
for $S,T\subseteq\Omega$. Given a code $\mathcal{C}\subseteq P(\Omega)$, the 
\textit{dual code} $\mathcal{C}^\perp$ is given by
\begin{equation*}
 \mathcal{C}^\perp =\lbrace S\subseteq\Omega\,\vert\,\vert S\cap T\vert\in 
2\mathbb{Z}\,\,\mathrm{for}\,\,\mathrm{any}\,\,T\in\mathcal{C}\rbrace.
\end{equation*}
A code $\mathcal{C}$ is self-dual if $\mathcal{C}=\mathcal{C}^\perp$. We say 
that a code $\mathcal{C}$ is a Type II binary linear code if 
$\vert\Omega\vert\in 4\mathbb{Z}$, $\Omega\in\mathcal{C}$, and $\vert T\vert\in 
4\mathbb{Z}$ for any $T\in\mathcal{C}$. Note that if $\Omega\in\mathcal{C}$, 
then $\mathcal{C}$ is closed under complements.

We now take $n\geq 1$ and consider $L(\frac{1}{2},0)^{\otimes n}$, and we 
identify $\Omega =\lbrace 1,\ldots,n\rbrace$. Consider a code 
$\mathcal{C}\subseteq P(\Omega)$ and for any $T\in\mathcal{C}$, we define 
$\omega_T\in L(\frac{1}{2},0)^{\otimes n}$ by
\begin{equation*}
 \omega_T=\sum_{i\notin T}\omega^{(i)}-\sum_{i\in T}\omega^{(i)}.
\end{equation*}
The main result of this section is:
\begin{theo}\label{goodform}
 Suppose $n\in 4\mathbb{Z}$, $\Omega =\lbrace 1,\ldots,n\rbrace$, and 
$\mathcal{C}\subseteq \mathcal{E}(\Omega)$ is a binary linear code satisfying 
$\Omega\in\mathcal{C}$ and for any distinct $i,j\in\Omega$ there is some 
$T_{ij}\in\mathcal{C}$ such that $i\in T_{ij}$ and $j\notin T_{ij}$. Then the 
vertex subring $L(\frac{1}{2},0)^{\otimes n}_\mathcal{C}$ generated by the 
vectors $\omega_T$ for $T\in\mathcal{C}$ is an integral form of 
$L(\frac{1}{2},0)^{\otimes n}$ containing $\omega$.
\end{theo}
\begin{proof}
 It is clear that $\omega\in L(\frac{1}{2},0)^{\otimes n}_\mathcal{C}$ since 
$\omega=\omega_{\emptyset}$. Since the vectors $\omega^{(i)}$ are homogeneous 
of 
conformal weight $2$, so are the vectors $\omega_T$ for $T\in\mathcal{C}$; then 
the standard fact that for homogeneous vectors $u$ and $v$ in a vertex operator 
algebra,
\begin{equation}\label{homops}
 \mathrm{wt}\,u_n v =\mathrm{wt}\,u+\mathrm{wt}\,v-n-1
\end{equation}
for $n\in\mathbb{Z}$, together with Proposition \ref{zgen}, implies that 
$L(\frac{1}{2},0)^{\otimes n}_\mathcal{C}$ is compatible 
with the conformal weight gradation of $L(\frac{1}{2},0)^{\otimes n}$ in the 
sense of (\ref{compatibility}).

We will show that the $\mathbb{C}$-span of 
$L(\frac{1}{2},0)^{\otimes n}_\mathcal{C}$ equals $L(\frac{1}{2},0)^{\otimes 
n}$ 
by showing that $L(\frac{1}{2},0)^{\otimes n}_\mathcal{C}$ contains non-zero multiples 
of 
the generators $\omega^{(i)}$ of $L(\frac{1}{2},0)^{\otimes n}$. By 
compatibility, this will show 
that the intersection of $L(\frac{1}{2},0)^{\otimes n}_\mathcal{C}$ with each 
weight space $L(\frac{1}{2},0)^{\otimes n}_{(m)}$ for $m\in\mathbb{Z}$ spans 
$L(\frac{1}{2},0)^{\otimes n}_{(m)}$. Then it will suffice to show that for any 
$m\in\mathbb{Z}$, $L(\frac{1}{2},0)^{\otimes n}_\mathcal{C}\cap 
L(\frac{1}{2},0)^{\otimes n}_{(m)}$ is a finitely generated abelian group, that 
is, a lattice, and that this lattice has rank equal to the dimension of 
$L(\frac{1}{2},0)^{\otimes n}_{(m)}$.

In fact, such a lattice must have rank at 
least the dimension of $L(\frac{1}{2},0)^{\otimes n}_{(m)}$ because 
$L(\frac{1}{2},0)^{\otimes n}_\mathcal{C}\cap 
L(\frac{1}{2},0)^{\otimes n}_{(m)}$ spans $L(\frac{1}{2},0)^{\otimes n}_{(m)}$ 
over $\mathbb{C}$. But also $L(\frac{1}{2},0)^{\otimes 
n}_\mathcal{C}\subseteq L(\frac{1}{2},0)^{\otimes n}_\mathbb{Q}$, where 
$L(\frac{1}{2},0)_\mathbb{Q}$ is the $\mathbb{Q}$-form of Proposition 
\ref{virmodqform}; this means that if $L(\frac{1}{2},0)^{\otimes 
n}_\mathcal{C}\cap 
L(\frac{1}{2},0)^{\otimes n}_{(m)}$ is a lattice, its rank must be no more than 
the dimension of $L(\frac{1}{2},0)^{\otimes n}_{(m)}$: any set of vectors in 
$L(\frac{1}{2},0)^{\otimes n}_\mathcal{C}\cap 
L(\frac{1}{2},0)^{\otimes n}_{(m)}$ which are linearly independent over $\Z$ are 
also linearly independent over $\mathbb{Q}$, since a dependence relation over 
$\mathbb{Q}$ reduces to a dependence relation over $\mathbb{Z}$ by clearing 
denominators. Thus it will suffice to show that  $L(\frac{1}{2},0)^{\otimes 
n}_\mathcal{C}\cap L(\frac{1}{2},0)^{\otimes n}_{(m)}$ is a lattice in 
$L(\frac{1}{2},0)^{\otimes n}_{(m)}$.
 
 To show that $L(\frac{1}{2},0)^{\otimes n}_\mathcal{C}$ contains multiples of 
$\omega^{(i)}$ for $1\leq i\leq n$, we claim that for any $i$,
 \begin{equation*}
  \dfrac{\vert\mathcal{C}\vert}{2}\omega^{(i)}=\sum_{T\in\mathcal{C},\,i\notin 
T} \omega_T\in L(\frac{1}{2},0)^{\otimes n}_\mathcal{C}.
 \end{equation*}
To prove this claim, consider any $j\neq i$, and let $\mathcal{C}_{ij}$ denote 
the one-dimensional code spanned by $\lbrace i,j\rbrace$. The dual code 
$\mathcal{C}_{ij}^\perp$ has dimension $n-1$ because it contains $2^{n-2}$ sets 
containing both $i$ and $j$ and $2^{n-2}$ sets containing neither $i$ nor $j$. 
Since $T_{ij}\in\mathcal{C}$ is not in $\mathcal{C}_{ij}^\perp$, 
$\mathcal{C}+\mathcal{C}_{ij}^\perp= P(\Omega)$, so that
\begin{equation*}
 \mathrm{dim}\,\mathcal{C}\cap\mathcal{C}_{ij}^\perp 
=\mathrm{dim}\,\mathcal{C}+\mathrm{dim}\,\mathcal{C}_{ij}^\perp-\mathrm{dim}\,
(\mathcal{C}+\mathcal{C}_{ij}^\perp)=\mathrm{dim}\,\mathcal{C}-1.
\end{equation*}
This means that half the sets in $\mathcal{C}$ contain neither or both of $i$ 
and $j$ and half contain exactly one of $i$ and $j$. Since 
$\Omega\in\mathcal{C}$, $\mathcal{C}$ is closed under complements. This means 
that of the sets in $\mathcal{C}$ which do not contain $i$, half contain $j$ 
and 
half do not. By the definition of the vectors $\omega_T$ for $T\in\mathcal{C}$, 
this proves the claim.

Now to finish the proof of the theorem, we just need to show that for any 
$m\in\mathbb{Z}$, $L(\frac{1}{2},0)^{\otimes n}_\mathcal{C}\cap 
L(\frac{1}{2},0)^{\otimes n}_{(m)}$ is the $\mathbb{Z}$-span of finitely many 
vectors. For $i$ such that $1\leq i\leq n$, we write
\begin{equation*}
 Y(\omega^{(i)},x)=\sum_{n\in\mathbb{Z}} L^{(i)}(n) x^{-n-2},
\end{equation*}
so that for any $i$ and $n\in\mathbb{Z}$,
\begin{equation*}
 L^{(i)}(n)=1^{\otimes (i-1)}\otimes L(n)\otimes 1^{\otimes (n-i)};
\end{equation*}
note that for $m,n\in\mathbb{Z}$, $[L^{(i)}(m),L^{(j)}(n)]=0$ when $i\neq j$. 
Also for $T\in\mathcal{C}$, we write
\begin{equation*}
 Y(\omega_T,x)=\sum_{n\in\mathbb{Z}} L_T(n) x^{-n-2}.
\end{equation*}
By Proposition \ref{zgen}, $L(\frac{1}{2},0)^{\otimes n}_\mathcal{C}$ is the 
$\mathbb{Z}$-span of vectors of the form
\begin{equation}\label{goodformbadspan}
 L_{T_1}(m_1)\cdots L_{T_k}(m_k)\mathbf{1}
\end{equation}
for $T_i\in\mathcal{C}$ and $m_i\in\mathbb{Z}$.

We now compute the following commutators for $S,T\in\mathcal{C}$ and 
$m,n\in\mathbb{Z}$:
\begin{align}\label{intformcomms}
 [L_S(m),L_T(n)]= & \left[\sum_{i\notin S} L^{(i)}(m)-\sum_{i\in S} 
L^{(i)}(m),\sum_{j\notin T} L^{(j)}(n)-\sum_{j\in T} 
L^{(j)}(n)\right]\nonumber\\
 = & \sum_{i\notin S+T} [L^{(i)}(m),L^{(i)}(n)]-\sum_{i\in S+T} 
[L^{(i)}(m),L^{(i)}(n)]\nonumber\\
 = &(m-n)\left(\sum_{i\notin S+T} L^{(i)}(m+n)-\sum_{i\in S+T} 
L^{(i)}(m+n)\right)\nonumber\\
 &+\dfrac{((n-\vert S+T\vert)-\vert 
S+T\vert)(m^3-m)}{24}\delta_{m+n,0}\nonumber\\
 = & (m-n)L_{S+T}(m+n)+\frac{n-2\vert S+T\vert}{4}\binom{m+1}{3}\delta_{m+n,0}.
\end{align}
Since $n\in 4\mathbb{Z}$ and $\mathcal{C}\subseteq\mathcal{E}(\Omega)$, and 
since $L_T(n)\mathbf{1}=0$ for $n\geq -1$, we can use these commutator 
relations 
to straighten products of the form (\ref{goodformbadspan}), so that  
$L(\frac{1}{2},0)^{\otimes n}_\mathcal{C}$ is the $\mathbb{Z}$-span of vectors 
of the form
\begin{equation*}
 L_{T_1}(-n_1)\cdots L_{T_k}(-n_k)\mathbf{1}
\end{equation*}
where $T_i\in\mathcal{C}$ and $n_1\geq\cdots\geq n_k>1$. Since $\omega_T$ has 
conformal weight $2$, $L_T(-n)$ raises conformal weight by $n$, using 
\eqref{homops}. Since also 
$\mathcal{C}$ is finite, this shows that for any $m\in\mathbb{Z}$ only finitely 
many vectors are required to span $L(\frac{1}{2},0)^{\otimes n}_\mathcal{C}\cap 
L(\frac{1}{2},0)^{\otimes n}_{(m)}$, as desired.
\end{proof}
\begin{rema}\label{redundant}
 Note that the generating set $\lbrace\omega_T\,\vert\, T\in\mathcal{C}\rbrace$ 
for 
$L(\frac{1}{2},0)^{\otimes n}_\mathcal{C}$ is redundant since if $T^c$ is the 
complement of $T\in\mathcal{C}$, then $\omega_{T^c}=-\omega_T$.
\end{rema}

\begin{rema}
 The assumption that $\Omega\in\mathcal{C}$, although convenient for the proof of Theorem \ref{goodform}, is not necessary since Remark \ref{redundant} shows that for any code $\mathcal{C}$, $L(\frac{1}{2},0)^{\otimes n}_\mathcal{C}=L(\frac{1}{2},0)^{\otimes n}_{\mathcal{C}+\mathbb{F}_2\Omega}$.
\end{rema}

We record two special cases of Theorem \ref{goodform} in the following 
corollary:
\begin{corol}\label{codespecialcases}
 If $n\in 4\mathbb{Z}$ and $\mathcal{C}\subseteq P(\Omega)$ equals 
$\mathcal{E}(\Omega)$ or is a Type II self-dual code, then 
$L(\frac{1}{2},0)^{\otimes n}_\mathcal{C}$ is an integral form of 
$L(\frac{1}{2},0)^{\otimes n}$.
\end{corol}
\begin{proof}
 Clearly $\Omega\in\mathcal{E}(\Omega)$ and $\Omega$ is in every Type II code 
by 
definition. Thus to apply Theorem \ref{goodform} we just need to show that for 
any distinct $i,j\in\Omega$, there is some $T_{ij}\in\mathcal{C}$ such that 
$i\in T_{ij}$ and $j\notin T_{ij}$. If $\mathcal{C}=\mathcal{E}(\Omega)$, then 
we can take $T_{ij}=\lbrace i,k\rbrace$ where $k$ is distinct from $i$ and $j$ 
(such a $k$ exists because $n\in 4\mathbb{Z}$). 
 
 If $\mathcal{C}$ is a Type II self-dual code, suppose for some $i$ and 
$j$ the desired $T_{ij}$ does not exist. Then since $\Omega\in\mathcal{C}$, 
$\mathcal{C}$ is closed under complements and every set in $\mathcal{C}$ 
contains both or neither of $i$ and $j$. This means $\lbrace 
i,j\rbrace\in\mathcal{C}^\perp=\mathcal{C}$, which is a contradiction since 
$\vert T\vert\in 4\mathbb{Z}$ if $T$ is a set in a Type II code. Hence the 
desired $T_{ij}$ exists.
\end{proof}

\begin{exam}
 We recall the smallest non-trivial Type II self-dual binary linear code, the 
Hamming code $\mathcal{H}$ on a set $\Omega$ of $8$ elements (see \cite{MS} 
Chapter 1 or \cite{FLM2} Chapter 10). The Hamming code can be described in 
several ways, but if we identify $\Omega=\lbrace 1,2,\ldots, 8\rbrace$, we can 
realize $\mathcal{H}$ explicitly as the sets
 \begin{align*}
  & \emptyset,\,\,\lbrace 1,2,3,4\rbrace,\,\,\lbrace 1,2,5,6\rbrace,\,\,\lbrace 
1,2,7,8\rbrace,\\
  & \lbrace 1,3,5,7\rbrace,\,\,\lbrace 2,4,5,7\rbrace,\,\,\lbrace 
2,3,6,7\rbrace,\,\,\lbrace 2,3,5,8\rbrace
 \end{align*}
and their complements.
\end{exam}

\section{Integral forms in modules for $L(\frac{1}{2},0)^{\otimes n}$}

We start this section by showing the existence of integral forms in certain 
modules for $L(\frac{1}{2},0)^{\otimes n}$ for $n\in 4\mathbb{Z}$. We use an 
integral form $L(\frac{1}{2},0)^{\otimes n}_\mathcal{C}$ where $\mathcal{C}$ is 
a code satisfying the conditions of Theorem \ref{goodform}. Recall that the 
irreducible $L(\frac{1}{2},0)^{\otimes n}$-modules are given by
\begin{equation*}
 W_H=L(\frac{1}{2},h_1)\otimes\cdots\otimes L(\frac{1}{2},h_n)
\end{equation*}
where $H=( h_1,\ldots, h_n)\in\lbrace 
0,\frac{1}{2},\frac{1}{16}\rbrace^n$.

\begin{propo}\label{virmodgoodform}
 Suppose $n\in 4\mathbb{Z}$ and $\mathcal{C}\subseteq\mathcal{E}(\Omega)$ is a 
code satisfying the conditions of Theorem \ref{goodform}. Moreover, suppose 
$H=( h_1,\ldots, h_n)\in\lbrace 0,\frac{1}{2}\rbrace^n$ is such that 
$h_i=\frac{1}{2}$ for an even number of $i$. If $v_H= 
v_{h_1}\otimes\cdots\otimes v_{h_n}$ where $v_{h_i}$ spans the lowest conformal 
weight space of $L(\frac{1}{2},h_i)$, then the $L(\frac{1}{2},0)^{\otimes 
n}_\mathcal{C}$-module $W_{H,\mathcal{C}}$ generated by $v_H$ is an integral 
form of $W_H$.
\end{propo}
\begin{proof}
 By Proposition \ref{tensorgens}, $W_H$ is generated as an 
$L(\frac{1}{2},0)^{\otimes n}$-module by $v_H$, so the 
$L(\frac{1}{2},0)^{\otimes n}_\mathcal{C}$-module generated by $v_H$ spans 
$W_H$ 
over $\mathbb{C}$. Also, $W_{H,\mathcal{C}}$ is compatible with the conformal 
weight grading of $W_H$ because $v_H$ is homogeneous of weight $h_1+\ldots+h_n$ 
and the generators $\omega_T$ for $T\in\mathcal{C}$ of 
$L(\frac{1}{2},0)^{\otimes n}_\mathcal{C}$ are homogeneous of weight $2$. Thus, 
as in the proof of Theorem \ref{goodform}, we just need to show that for any 
conformal weight $n\in\mathbb{C}$, $W_{H,\mathcal{C}}\cap (W_H)_{(n)}$ is 
spanned over $\mathbb{Z}$ by finitely many vectors.
 
 By Proposition \ref{zgen} and the commutation relations (\ref{intformcomms}), 
$W_{H,\mathcal{C}}$ is the $\mathbb{Z}$-span of products of the form
 \begin{equation}\label{modzformspan}
  L_{T_1}(-n_1)\cdots L_{T_k}(-n_k) v_H
 \end{equation}
for $T_i\in\mathcal{C}$ and $n_1\geq\ldots\geq n_k\geq 0$, as in the proof 
of Theorem \ref{goodform}. Now suppose that $S\subseteq\Omega$ is the set such 
that $h_i=\frac{1}{2}$ for $i\in S$, so that by assumption $\vert S\vert\in 
2\mathbb{Z}$. Then for $T\in\mathcal{C}$,
\begin{equation*}
 L_T(0)v_H=\left(\dfrac{\vert S\vert-\vert S\cap T\vert}{2}-\dfrac{\vert S\cap 
T\vert}{2}\right) v_H=\left(\dfrac{\vert S\vert}{2}-\vert S\cap 
T\vert\right)v_H\in\mathbb{Z} v_H.
\end{equation*}
Thus $W_{H,\mathcal{C}}$ is in fact the $\mathbb{Z}$-span of products of the 
form (\ref{modzformspan}) where now $n_1\geq\ldots\geq n_k >0$, and we conclude 
that 
$W_{H,\mathcal{C}}\cap (W_H)_{(n)}$ is spanned over $\mathbb{Z}$ by finitely 
many vectors for any $n\in\mathbb{C}$, just as in the proof of Theorem 
\ref{goodform}.
\end{proof}

\begin{rema}
 Note that in the statement of Proposition \ref{virmodgoodform}, the condition 
that $h_i=\frac{1}{2}$ for an even number of $i$ is equivalent to the condition 
that $h_1+\ldots+h_n\in\Z$. In this case, all conformal weights of $W_H$ are integers.
\end{rema}

We see that if $H=( h_1,\ldots,h_n)\in\lbrace 0,\frac{1}{2},\frac{1}{16}\rbrace 
^n$ is such that $h_i=\frac{1}{2}$ for an odd number of $i$ or if 
$h_i=\frac{1}{16}$ for any $i$, the $L(\frac{1}{2},0)^{\otimes 
n}_\mathcal{C}$-module generated by $v_H$ will not usually be an integral form 
of $W_H$. This is because we will generally have $L_T(0)v_H\notin\mathbb{Z} 
v_h$ 
for $T\in\mathcal{C}$, even if $\mathcal{C}$ satisfies the conditions of 
Theorem 
\ref{goodform}. For the remainder of this section, we will concentrate on the 
case $n=16$, since the lattice vertex operator algebra $V_{E_8}$ based on the 
$E_8$ root lattice contains a subalgebra isomorphic to 
$L(\frac{1}{2},0)^{\otimes 16}$ (\cite{DMZ}); thus our work here suggests an 
alternate approach to obtaining integral structure in $V_{E_8}$, different from 
the approach to obtaining integral forms in lattice vertex operator algebras in 
\cite{B}, \cite{P}, \cite{DG}, and \cite{M2}. Since the moonshine module vertex 
operator algebra 
$V^\natural$ similarly contains a subalgebra isomorphic to 
$L(\frac{1}{2},0)^{\otimes 48}$ (\cite{DMZ}), the same ideas may lead to the 
construction of explicit integral structure in $V^\natural$.

The decomposition of $V_{E_8}$ into submodules for $L(\frac{1}{2},0)^{\otimes 
16}$ from \cite{DMZ} is as follows:
\begin{equation*}
 V_{E_8}=\coprod_{\substack{H=( h_1,\ldots,h_{16})\in\lbrace 
0,\frac{1}{2}\rbrace^{16}\\ h_1+\ldots+h_{16}\in\mathbb{Z}}} W_H\oplus 2^7 
W_{(\frac{1}{16},\cdots,\frac{1}{16})}.
\end{equation*}
\begin{rema}
 Actually, according to \cite{DMZ}, the multiplicity of 
$W_{(\frac{1}{16},\cdots,\frac{1}{16})}$ in $V_{E_8}$ is $2^8$. 
However, this seems to be an error; there are $\binom{16}{2}=120$ submodules 
$W_H$ where $H=( h_1,\ldots,h_{16})\in\lbrace 0,\frac{1}{2}\rbrace^{16}$ is such 
that $h_1+\ldots+h_{16}=1$, 
and then $\binom{16}{2}+2^7=248$ gives the correct dimension of 
$(V_{E_8})_{(1)}$.
\end{rema}
\noindent Since we have already obtained integral structure in the modules 
$W_H$ 
for which 
each $h_i\in\lbrace 0,\frac{1}{2}\rbrace$ and $h_1+\ldots+h_{16}\in\mathbb{Z}$, 
we now find a code $\mathcal{C}_{16}$ on a $16$-element set such 
that $W_{(\frac{1}{16},\cdots,\frac{1}{16})}$ has an integral form 
generated as an $L(\frac{1}{2},0)^{\otimes 16}_{\mathcal{C}_{16}}$-module by 
$v_{(\frac{1}{16},\cdots,\frac{1}{16})}$. 

To construct $\mathcal{C}_{16}$, we consider an $8$-element set $\Omega=\lbrace 
1,2,\ldots, 8\rbrace$ and an $8$-element set $\Omega'=\lbrace 
1',2',\ldots,8'\rbrace$. For any $T\subseteq\Omega$, we use $T'$ to denote the 
corresponding set in $\Omega'$. Suppose $\mathcal{H}$ is a Hamming code in 
$P(\Omega)$. Then we define $\mathcal{C}_{16}$ to be the code in 
$P(\Omega\cup\Omega')$ generated as an $\mathbb{F}_2$-vector space by the sets 
$T\cup T'$ for $T\in\mathcal{H}$ and $\Omega'$. Thus $\mathcal{C}_{16}$ is the 
$5$-dimensional code consisting of the $32$ sets $T\cup T'$ and $T\cup (T')^c$ 
for $T\in\mathcal{H}$.
\begin{propo}\label{c16code}
 The binary linear code $\mathcal{C}_{16}$ satisfies the conditions of Theorem 
\ref{goodform}, and for any $T\in\mathcal{C}_{16}$, $\vert T\vert\in 
8\mathbb{Z}$.
\end{propo}
\begin{proof}
 We have $\Omega\cup\Omega'\in\mathcal{C}_{16}$ since $\Omega\in\mathcal{H}$. 
We 
need to show that for any distinct $i,j\in\Omega\cup\Omega'$, there is some 
$T_{ij}\in\mathcal{C}_{16}$ such that $i\in T_{ij}$ but $j\notin T_{ij}$. Now 
for any $i\in\Omega$ and $j\neq i,i'$, there is some $T\in\mathcal{H}$ such 
that 
$i\in T$ but $j\notin T$ (if $j\in\Omega$) or $j\notin T'$ (if $j\in\Omega'$), 
using Corollary \ref{codespecialcases} since $\mathcal{H}$ is a Type II 
self-dual code. In this case we can take $T_{ij}=T\cup T'$. If $j=i'$, then 
$i\in T$ 
for some $T\in\mathcal{H}$, and we can take $T_{ij}=T\cup (T')^c$. The proof if 
$i\in\Omega'$ is the same.
 
 To prove the second assertion of the proposition, we note that for any 
$T\in\mathcal{H}$, $\vert T\vert\in 4\mathbb{Z}$. If $\vert T\vert=0$, then 
$\vert T\cup T'\vert=0$ and $\vert T\cup (T')^c\vert=8$. If $\vert T\vert =4$, 
then both $\vert T\cup T'\vert=8$ and $\vert T\cup (T')^c\vert=8$. If $\vert 
T\vert =8$, then $\vert T\cup T'\vert =16$ and $\vert T\cup (T')^c\vert=8$.
\end{proof}

\begin{propo}\label{modc16}
 Suppose $n=16$ and $H=(\frac{1}{16},\cdots,\frac{1}{16})$. Then the 
$L(\frac{1}{2},0)^{\otimes 16}_{\mathcal{C}_{16}}$-submodule 
$W_{H,\mathcal{C}_{16}}$ of $W_H$ generated by $v_H$ is an integral form of 
$W_H$.
\end{propo}
\begin{proof}
 The proof is the same as the proof of Proposition \ref{virmodgoodform}, except 
that for $T\in\mathcal{C}_{16}$,
 \begin{equation*}
  L_T(0)v_H=\left(\dfrac{16-\vert T\vert}{16}-\dfrac{\vert 
T\vert}{16}\right)v_H=\dfrac{16-2\vert T\vert}{16} v_H\in\mathbb{Z}v_H
 \end{equation*}
because $\vert T\vert\in 8\mathbb{Z}$ by Proposition \ref{c16code}.
\end{proof}

\section{Contragredients and integral intertwining operators}

In this section, we fix $n\in 4\mathbb{Z}$ and a code 
$\mathcal{C}\subseteq\mathcal{E}(\Omega)$ on an $n$-element set $\Omega$ 
satisfying the conditions of Theorem \ref{goodform}. 
We will consider the contragredients of and intertwining operators among 
$L(\frac{1}{2},0)^{\otimes n}_\mathcal{C}$-modules.

We first observe that since $L(\frac{1}{2},0)_{(0)}=\mathbb{C}\mathbf{1}$, and 
$L(\frac{1}{2},0)_{(1)}=0$, \cite{Li} implies that $L(\frac{1}{2},0)$ has a 
one-dimensional space of invariant bilinear forms $(\cdot,\cdot)$; since in 
addition $L(\frac{1}{2},0)$ is a simple vertex operator algebra, any non-zero 
such form is nondegenerate. In particular, $L(\frac{1}{2},0)$ is 
self-contragredient. In fact, any irreducible $L(\frac{1}{2},0)$-module is 
self-contragredient, since the contragredient of an irreducible module is 
irreducible and the contragredient of a module has the same conformal weights 
as 
the module. (Note that the conformal weights of $L(\frac{1}{2},h)$ lie in 
$h+\mathbb{N}$ for $h=0,\frac{1}{2},\frac{1}{16}$.) 

For $h=0,\frac{1}{2},\frac{1}{16}$, we let $(\cdot,\cdot)_h$ denote the 
invariant bilinear form on $L(\frac{1}{2},h)$ such that $(v_h,v_h)=1$, where 
$v_h$ is a lowest conformal weight vector generating $L(\frac{1}{2},h)$ and 
$v_0=\mathbf{1}$. We note that for any $n\geq 1$, irreducible modules for 
$L(\frac{1}{2},0)^{\otimes n}$ are self-contragredient since for any $H=( 
h_1,\ldots h_n)\in\lbrace 0,\frac{1}{2},\frac{1}{16}\rbrace^n$, the module 
$W_H=L(\frac{1}{2},h_1)\otimes\cdots\otimes L(\frac{1}{2},h_n)$ has a
nondegenerate invariant bilinear form $(\cdot,\cdot)_H$ determined by
\begin{equation*}
 (v_1\otimes\cdots\otimes v_n,w_1\otimes\cdots\otimes w_n)_H=\prod_{i=1}^n 
(v_i,w_i)_{h_i},
\end{equation*}
where $v_i,w_i\in L(\frac{1}{2},h_i)$ for $i=1,\ldots ,n$.

Now take $\Omega=\lbrace 1,\ldots, n\rbrace$ and suppose 
$\mathcal{C}\subseteq\mathcal{E}(\Omega)$ is a code satisfying the conditions 
of 
Theorem \ref{goodform}. Since $L(1)\omega_T=0$ for any $T\in\mathcal{C}$, 
Propositions \ref{contmod} and \ref{invarunderl1} immediately imply:
\begin{propo}
 Graded $\Z$-duals of $L(\frac{1}{2},0)^{\otimes n}_\mathcal{C}$-modules are 
$L(\frac{1}{2},0)^{\otimes n}_\mathcal{C}$-modules.
\end{propo}
\noindent Moreover, if $W_H$ is an irreducible $L(\frac{1}{2},0)^{\otimes 
n}$-module such 
that the $L(\frac{1}{2},0)^{\otimes n}_\mathcal{C}$-module $W_{H,\mathcal{C}}$ 
generated by a lowest conformal weight vector $v_H$ is an integral form of 
$W_H$, then we can use the form $(\cdot,\cdot)_H$ to identify 
$W_{H,\mathcal{C}}'$ with another, generally different, integral form of $W_H$.

We now consider intertwining operators among $L(\frac{1}{2},0)^{\otimes 
n}$-modules and state the main result of this section:
\begin{theo}\label{virintwop}
 Suppose $n\in 4\mathbb{Z}$, $\Omega=\lbrace 1,\ldots, n\rbrace$, and 
$\mathcal{C}\subseteq\mathcal{E}(\Omega)$ is a code satisfying the conditions 
of 
Theorem \ref{goodform}. In addition, suppose $W_{H^{(i)}}$ for $i=1,2,3$ are 
irreducible $L(\frac{1}{2},0)^{\otimes n}$-modules such that the 
$L(\frac{1}{2},0)^{\otimes n}_\mathcal{C}$-modules $W_{H^{(i)},\mathcal{C}}$ 
generated by a lowest conformal weight vector $v_{H^{(i)}}$ are integral forms 
of $W_{H^{(i)}}$. Then an intertwining operator $\mathcal{Y}$ of type 
$\binom{W_{H^{(3)}}}{W_{H^{(1)}}\,W_{H^{(2)}}}$ is integral with respect to 
$W_{H^{(1)},\mathcal{C}}$, $W_{H^{(2)},\mathcal{C}}$, and 
$W_{H^{(3)},\mathcal{C}}'$ if and only if the coefficient of 
$\mathcal{Y}(v_{H^{(1)}},x)v_{H^{(2)}}$ in the lowest weight space of 
$W_{H^{(3)}}$ is in $\mathbb{Z} v_{H^{(3)}}$, where we 
identify $W_{H^{(3)}}$ with $W_{H^{(3)}}'$ using $(\cdot,\cdot)_{H^{(3)}}$.
\end{theo}
We first prove the following general lemma:
\begin{lemma}\label{crossbracketslemma}
Suppose $n\in\mathbb{Z}$, $\Omega=\lbrace 1,\ldots,n\rbrace$ and 
$\mathcal{C}\subseteq P(\Omega)$ is a code. If $\mathcal{Y}$ is an intertwining 
operator among $L(\frac{1}{2},0)^{\otimes n}$-modules $W_{H^{(i)}}$ for 
$i=1,2,3$, then
\begin{align}\label{crossbrackets}
 [L_T(m),\mathcal{Y}(v_{H^{(1)}},x)] & 
=x^{m}[L_T(0),\mathcal{Y}(v_{H^{(1)}},x)]+mx^{m}\mathcal{Y}(L_T(0)v_{H^{(1)}},
x)
\end{align}
for $T\in\mathcal{C}$ and $m\geq 0$.
\end{lemma}
\begin{proof}
The coefficient of $x_0^{-2}$ in the Jacobi identity for intertwining operators 
(see for example formula (5.4.4) in \cite{FHL}) implies that for 
$T\in\mathcal{C}$, 
\begin{align*}
 (x_1-x_2)[Y(\omega_T,x_1),\mathcal{Y}(  v_{H^{(1)}} & ,x_2)]  
=\mathrm{Res}_{x_0} x_0 
x_2^{-1}\delta\left(\dfrac{x_1-x_0}{x_2}\right)\mathcal{Y}(Y(\omega_T,x_0)v_{H^{
(1)}},x_2)\\
 & =\sum_{i\geq 0}\dfrac{(-1)^i}{i!}\left(\dfrac{\partial}{\partial 
x_1}\right)^i\left(x_2^{-1}\delta\left(\dfrac{x_1}{x_2}\right)\right)\mathcal{Y}
(L_T(i) v_{H^{(1)}},x_2)\\
 & = 
x_2^{-1}\delta\left(\dfrac{x_1}{x_2}\right)\mathcal{Y}(L_T(0)v_{H^{(1)}},x_2),
\end{align*}
since $L_T(i)v_{H^{(1)}}=0$ for $i>0$. Now taking the coefficient of 
$x_1^{-m-2}$ and replacing $x_2$ with $x$, we obtain
\begin{equation}\label{crossbrackets2}
 [L_T(m+1),\mathcal{Y}(v_{H^{(1)}},x)]-x [L_T(m),\mathcal{Y}(v_{H^{(1)}},x)] = 
x^{m+1}\mathcal{Y}(L_T(0)v_{H^{(1)}},x).
\end{equation}

 We now prove (\ref{crossbrackets}) by induction on $m$. It is certainly true 
for $m=0$, and assuming it is true for $m\geq 0$, we see from 
(\ref{crossbrackets2}) that
 \begin{align*}
  [& L_T(m+1)  ,\mathcal{Y}(v_{H^{(1)}},x)]  = 
x[L_T(m),\mathcal{Y}(v_{H^{(1)}},x)]+x^{m+1}\mathcal{Y}(L_T(0)v_{H^{(1)}},x)\\
  & =x(x^m[L_T(0),\mathcal{Y}(v_{H^{(1)}},x)]+m 
x^{m}\mathcal{Y}(L_T(0)v_{H^{(1)}},x))+x^{m+1}\mathcal{Y}(L_T(0)v_{H^{(1)}},
x)\\
  & = 
x^{m+1}[L_T(0),\mathcal{Y}(v_{H^{(1)}},x)]+(m+1)x^{m+1}\mathcal{Y}(L_T(0)v_{
H^{(1)}},x),
 \end{align*}
proving the lemma.
\end{proof}

\begin{rema}
 Considering the coefficient of $x_0^{-2}$ in the Jacobi identity for intertwining operators as in the proof of Lemma \ref{crossbrackets} is motivated by the cross-brackets of vertex operators discussed in \cite{FLM2} Section 8.9. If $u$ and $v$ are two vectors in a vertex operator algebra $V$, then the crossbracket of $u$ and $v$ is the collection of operators
 \begin{equation*}
  [u\times_1 v]_{m,n}=[u_{m+1}, v_n]-[u_m, v_{n+1}]
 \end{equation*}
for $m,n\in\Z$. Cross-brackets are particularly natural to consider when $u$ and $v$ have conformal weight $2$ and when $V$ has no non-zero elements of conformal weight $1$ or of weight less than $0$. Then the coefficient of $x_0^{-2}$ in the Jacobi identity for vertex operator algebras shows that
\begin{equation}\label{commaff}
[u\times_1 v]_{m,n}=(u_1 v)_{m+n}+\dfrac{m(m-1)}{2}\delta_{m+n,1} u_3 v.
\end{equation}
for $m,n\in\Z$. In this case, $(u,v)\rightarrow u_1 v$ defines a commutative product on $V_{(2)}$ and if $V_{(0)}=\mathbb{C}\mathbf{1}$, then $(u,v)\rightarrow u_3 v$ defines a symmetric bilinear form on $V_{(2)}$. Note that \eqref{commaff} is analogous to the commutation relations in an affine Lie algebra.
\end{rema}

We now proceed with the proof of Theorem \ref{virintwop}:
\begin{proof}
 Since we identify $W_{H^{(3)}}$ with $W_{H^{(3)}}'$ using 
$(\cdot,\cdot)_{H^{(3)}}$ and since $(v_{H^{(3)}},v_{H^{(3)}})_{H^{(3)}}=1$, it 
is clear that if $\mathcal{Y}$ is integral with respect to 
$W_{H^{(1)},\mathcal{C}}$, $W_{H^{(2)},\mathcal{C}}$, and 
$W_{H^{(3)},\mathcal{C}}'$, then the coefficient of 
$\mathcal{Y}(v_{H^{(1)}},x)v_{H^{(2)}}$ in the lowest weight space of 
$W_{H^{(3)}}$ must be in $\mathbb{Z}v_{H^{(3)}}$.

Conversely, suppose that the coefficient of 
$\mathcal{Y}(v_{H^{(1)}},x)v_{H^{(2)}}$ in the lowest weight space of 
$W_{H^{(3)}}$ is in $\mathbb{Z}v_{H^{(3)}}$, that is,
\begin{equation}\label{incontraassump}
 (\mathcal{Y}(v_{H^{(1)}},x)v_{H^{(2)}}, 
v_{H^{(3)}})_{H^{(3)}}\in\mathbb{Z}\lbrace x\rbrace.
\end{equation}
By Theorem \ref{intwopgen}, it is enough to prove that
\begin{equation}\label{incontra}
 (\mathcal{Y}(v_{H^{(1)}},x)v_{H^{(2)}}, w)_{H^{(3)}}\in\mathbb{Z}\lbrace 
x\rbrace
\end{equation}
for any $w\in W_{H^{(3)},\mathcal{C}}$. To prove this, we use induction on the 
conformal weight of $w$ (which is contained in $\sum_{i=1}^n 
h^{(3)}_i+\mathbb{N}$), the base case given by (\ref{incontraassump}). Since we 
see from the proofs of Propositions \ref{virmodgoodform} and \ref{modc16} that 
$W_{H^{(3)}}$ is the $\mathbb{Z}$-span of vectors of the form
\begin{equation*}
 L_{T_1}(-m_1)\cdots L_{T_k}(-m_k) v_{H^{(3)}}
\end{equation*}
where $m_i>0$, it is enough to show that if (\ref{incontra}) holds for $w\in 
W_{H^{(3)},\mathcal{C}}$ of weight less than some fixed $N$, then it also holds 
for $L_T(-m)w$ for any $T\in\mathcal{C}$ and $m>0$. 

To prove this, we first observe that since
\begin{equation*}
 Y(e^{x L(1)}(-x^{-2})^{L(0)}\omega_T,x^{-1})=x^{-4} 
Y(\omega_T,x^{-1}),
\end{equation*}
the operator $L_T(m)$ is adjoint to $L_T(-m)$ with respect to 
$(\cdot,\cdot)_{H^{(3)}}$ for any $T\in\mathcal{C}$. Then we use Lemma 
\ref{crossbracketslemma} and the fact that $L_T(m)v_{H^{(2)}}=0$ for $m>0$ to 
obtain
\begin{align*}
 (\mathcal{Y}(v_{H^{(1)}},x)&v_{H^{(2)}}, L_T(-m)w)_{H^{(3)}}  
=(L_T(m)\mathcal{Y}(v_{H^{(1)}},x)v_{H^{(2)}},w)_{H^{(3)}}\\
 & = ([L_T(m),\mathcal{Y}(v_{H^{(1)}},x)]v_{H^{(2)}},w)_{H^{(3)}}\\
 & = (x^m 
[L_T(0),\mathcal{Y}(v_{H^{(1)}},x)]v_{H^{(2)}}+mx^{m}\mathcal{Y}(L_T(0)v_{H^{
(1)}},x)v_{H^{(2)}},w)_{H^{(3)}}\\
 & 
=x^m(\mathcal{Y}(v_{H^{(1)}},x)v_{H^{(2)}},L_T(0)w)_{H^{(3)}}-x^m(\mathcal{Y}(v_
{H^{(1)}},x)L_T(0)v_{H^{(2)}},w)_{H^{(3)}}\\
 &\,\,\,\,\,\,\,+m 
x^{m}(\mathcal{Y}(L_T(0)v_{H^{(1)}},x)v_{H^{(2)}},w)_{H^{(3)}}\in\mathbb{Z}
\lbrace x\rbrace
\end{align*}
because $L_T(0)v_{H^{(i)}}\in\mathbb{Z}v_{H^{(i)}}$ for $i=1,2$ and $L_T(0)w\in 
W_{H^{(3)},\mathcal{C}}$ has the same weight as $w$. This proves the theorem.
\end{proof}

\section{Application to framed vertex operator algebras}

We recall that in \cite{DGH}, vertex operator algebras which contain a vertex 
operator subalgebra isomorphic to $L(\frac{1}{2},0)^{\otimes n}$ for some 
$n\in\mathbb{N}$ are called framed vertex operator algebras. We conclude this 
paper by suggesting a strategy for obtaining interesting integral forms in a framed 
vertex operator algebra $V$, when $n\in 4\mathbb{Z}$. From \cite{W} and 
\cite{DMZ}, the $L(\frac{1}{2},0)^{\otimes 
n}$-module $V$ is completely reducible, and has a finite decomposition 
$V=\coprod W_H$ for some $H=( h_1,\ldots,h_n)\in\lbrace 
0,\frac{1}{2},\frac{1}{16}\rbrace^n$. If $\Omega=\lbrace 1,\ldots n\rbrace $ 
and 
$\mathcal{C}\subseteq\mathcal{E}(\Omega)$ is a code satisfying the conditions 
of 
Theorem \ref{goodform}, and if for any $W_H$ appearing in the decomposition of 
$V$, the $L(\frac{1}{2},0)^{\otimes n}_\mathcal{C}$-submodule 
$W_{H,\mathcal{C}}$ of $W_H$ generated by a lowest conformal weight vector 
$v_H$ 
is an integral form of $W_H$, then $V_\mathcal{C}=\coprod W_{H,\mathcal{C}}$ is 
an integral form of $V$ as a vector space. Note that we can take $v_{(0,\ldots,0)}=\mathbf{1}$, so that $W_{(0,\ldots,0),\mathcal{C}}= L(\frac{1}{2},0)^{\otimes n}_\mathcal{C}$. The bilinear forms $(\cdot,\cdot)_H$ induce a bilinear form $(\cdot,\cdot)$ on $V$, showing that $V\cong V'$ as an $L(\frac{1}{2},0)^{\otimes n}$-module (note that $(\cdot,\cdot)$ does not necessarily induce an isomorphism of $V$-modules). We use $V_\mathcal{C}'$ to denote the graded $\Z$-dual of $V_\mathcal{C}$ with respect to the form $(\cdot,\cdot)$.

We observe that if $W_{H^{(i)}}$ for $i=1,2,3$ are three submodules appearing 
in 
the decomposition of $V$, then the vertex operator $Y$ on $V$ restricted to 
$W_{H^{(1)}}\otimes W_{H^{(2)}}$ and projected onto $W_{H^{(3)}}$ is an 
intertwining operator of type $\binom{W_{H^{(3)}}}{W_{H^{(1)}}\,W_{H^{(2)}}}$. Then Theorem \ref{virintwop} implies
\begin{theo}
 Suppose that for any triple of submodules $W_{H^{(i)}}$ for $i=1,2,3$ appearing in the decomposition of $V$ as an $L(\frac{1}{2},0)^{\otimes n}$-module, the lowest weight vectors $v_{H^{(i)}}$ generating the $W_{H^{(i)},\mathcal{C}}$ as $L(\frac{1}{2},0)^{\otimes n}_\mathcal{C}$-modules satisfy
 \begin{equation*}
  (Y(v_{H^{(1)}},x)v_{H^{(2)}},v_{H^{(3)}})\in\Z\lbrace x\rbrace.
 \end{equation*}
Then for any $u,v\in V_\mathcal{C}$,
\begin{equation}\label{framedop}
 Y(u,x)v\in V_\mathcal{C}'[[x,x^{-1}]].
\end{equation}
\end{theo}
This theorem does not quite imply that $V_\mathcal{C}$ is an integral form of $V$ because we cannot assume $V_\mathcal{C}=V_\mathcal{C}'$. It is not clear in general when $Y$ induces intertwining operators which are integral with respect to 
$W_{H^{(1)},\mathcal{C}}$, $W_{H^{(2)},\mathcal{C}}$, and 
$W_{H^{(3)},\mathcal{C}}$ for any triple of submodules $W_{H^{(i)}}$, $i=1,2,3$, 
appearing in the decomposition of $V$. We end with an example that 
illustrates the issues involved:
\begin{exam}
 Let $Q$ be the $A_1\times A_1$ root lattice: 
$Q=\mathbb{Z}\alpha_1+\mathbb{Z}\alpha_2$ where 
$\langle\alpha_i,\alpha_j\rangle=2\delta_{ij}$. Consider the lattice vertex 
operator algebra $V_Q$ (see \cite{FLM2} or \cite{LL} for the construction of 
lattice vertex operator algebras and for the notation used below). For any $\alpha\in Q$, we define
 \begin{equation*}
  \iota(e_\alpha)^\pm =\iota(e_\alpha)\pm\iota(e_{-\alpha}).
 \end{equation*}

 The proof of Theorem 6.3 in \cite{DMZ} shows that $V_Q$ contains a vertex 
operator subalgebra isomorphic to $L(\frac{1}{2},0)^{\otimes 4}$, where
 \begin{align*}
  \omega^{(1)} & 
=\dfrac{1}{16}(\alpha_1+\alpha_2)(-1)^2\mathbf{1}+\dfrac{1}{4}\iota(e_{
\alpha_1+\alpha_2})^+,\\
  \omega^{(2)} & 
=\dfrac{1}{16}(\alpha_1+\alpha_2)(-1)^2\mathbf{1}-\dfrac{1}{4}\iota(e_{
\alpha_1+\alpha_2})^+,\\
  \omega^{(3)} & 
=\dfrac{1}{16}(\alpha_1-\alpha_2)(-1)^2\mathbf{1}+\dfrac{1}{4}\iota(e_{
\alpha_1-\alpha_2})^+,\\
  \omega^{(4)} & 
=\dfrac{1}{16}(\alpha_1-\alpha_2)(-1)^2\mathbf{1}-\dfrac{1}{4}\iota(e_{
\alpha_1-\alpha_2})^+.
 \end{align*}
It is not hard to check that the decomposition of $V_Q$ into 
$L(\frac{1}{2},0)^{\otimes 4}$-modules is given by
\begin{equation*}
 V_Q=\coprod_{\substack{H=( h_1,h_2,h_3,h_4)\in\lbrace 
0,\frac{1}{2}\rbrace^{4}\\ h_1+h_2+h_3+h_{4}\in\mathbb{Z}}} W_H,
\end{equation*}
where the lowest conformal weight vectors generating the submodules in the 
decomposition can be taken to be
\begin{equation*} 
\mathbf{1},\,\,\dfrac{\iota(e_{\alpha_1})^+\pm\iota(e_{\alpha_2})^+}{2},\,\,
\dfrac{\iota(e_{\alpha_1})^-\pm\iota(e_{\alpha_2})^-}{2},\,\,\dfrac{
(\alpha_1\pm\alpha_2)(-1)}{2}\mathbf{1},\,\,\dfrac{\alpha_1(-1)^2-\alpha_2(-1)^2
}{4}\mathbf{1}.
\end{equation*}
Setting $\Omega=\lbrace 1,2,3,4\rbrace$ and $\mathcal{C}=\mathcal{E}(\Omega)$, 
we take these vectors as the vectors $v_H$ generating 
$V_{Q,\mathcal{E}(\Omega)}$ as an $L(\frac{1}{2},0)^{\otimes 
4}_{\mathcal{E}(\Omega)}$-module. It is straightforward albeit tedious to 
calculate each $(Y(v_{H^{(1)}},x)v_{H^{(2)}}, v_{H^{(3)}})$, so we can show that $Y$ for $V_Q$ satisfies 
\eqref{framedop}.

However, it is not generally the case here that 
$W_{H,\mathcal{E}(\Omega)}=W_{H,\mathcal{E}(\Omega)}'$. For example, consider 
$H=(\frac{1}{2},\frac{1}{2},0,0)$. Then the weight $2$ subspace of 
$W_H$ is the $\mathbb{Z}$-span of the vectors $L_T(-1)v_H$ for 
$T\in\mathcal{E}(\Omega)$, which is the integral span of the vectors 
$(L^{(1)}(-1)\pm L^{(2)}(-1))v_H$. But
\begin{align*}
& \left((  L^{(1)}(-1)\pm L^{(2)}(-1))v_H, (L^{(1)}(-1)\pm L^{(2)}(-1))v_H \right)_H\\
&\;\;\;\;\;\; =\left(v_H,(L^{(1)}(1)\pm L^{(2)}(1))(L^{(1)}(-1)\pm L^{(2)}(-1))v_H\right)_H\\
&\;\;\;\;\;\; =\left(v_H,2(L^{(1)}(0)+L^{(2)}(0))v_H\right)_H=2.
\end{align*}
Also,
\begin{equation*}
 \left((L^{(1)}(-1)\pm L^{(2)}(-1))v_H, (L^{(1)}(-1)\mp L^{(2)}(-1))v_H\right)_H=0,
\end{equation*}
so that the weight $2$ subspace of $W_{H,\mathcal{E}(\Omega)}'$ is the integral 
span of the two vectors $\frac{1}{2}(L^{(1)}(-1)\pm L^{(2)}(-1))v_H$, showing 
that $(W_{H,\mathcal{E}(\Omega)}')_{(2)}\supsetneq 
(W_{H,\mathcal{E}(\Omega)})_{(2)}$.
\end{exam}

\noindent{\small \sc 
Beijing International Center for Mathematical Research, Peking University, Beijing, China 100084}\\
{\em E--mail address}:
\texttt{robertmacrae@math.pku.edu.cn} \\

\end{document}